\documentclass[12pt]{amsart}
\usepackage{amssymb, amsmath, amsthm, amsfonts, latexsym}

\usepackage{cite}
\usepackage{mathtools}
\usepackage[top=30mm, bottom=30mm, left=25mm, right=25mm]{geometry}

\theoremstyle{plain} 
\newtheorem{theorem}[equation]{Theorem}

\newtheorem{lemma}[equation]{Lemma}
\newtheorem{proposition}[equation]{Proposition}
\theoremstyle{definition}
\newtheorem{definition}[equation]{Definition}

\theoremstyle{remark}
\newtheorem{remark}[equation]{Remark}
\newtheorem{example}[equation]{Example}

\title{Heat semigroups on quantum automorphism groups of finite dimensional C*-algebras}
\author{Futaba Sato}
\address{Futaba Sato\\
Graduate School of Mathematical Sciences\\
The University of Tokyo\\
3-8-1 Komaba, Meguro-ku\\
Tokyo 153-8914, Japan}
\email{sato-futaba@g.ecc.u-tokyo.ac.jp}

\begin{document}
\maketitle

\begin{abstract}
In this paper, we investigate heat semigroups on a quantum automorphism group $\mathrm{Aut}^+(B)$ of a finite dimensional C*-algebra $B$ and its Plancherel trace. We show ultracontractivity, hypercontractivity, and the spectral gap inequality of the heat semigroups on ${\rm Aut}^+(B)$. Furthermore, we obtain the sharpness of the Sobolev embedding property and the Hausdorff-Young inequality of ${\rm Aut}^+(B)$.
\end{abstract}

\section{Introduction}
In this paper, we investigate heat semigroups on a class of compact quantum groups called quantum automorphism groups of finite dimensional C*-algebras, denoted by ${\rm Aut}^+(B,\psi)$ defined for a pair of a finite dimensional C*-algebra $B$ and a faithful state $\psi$ on $B$ introduced by Wang in 1998 \cite{Wan98}. 
This class contains quantum permutation groups $S_n^+$ and ``projective" versions of quantum orthogonal groups.

In the framework of compact quantum groups, the concept of Levy processes have a generalization in terms of semigroups of linear maps, and an analogue of Markov semigroups given by completely positive maps are particularly useful. 
According to \cite{CFK}, as compact quantum groups do not have a differential structure to define the Laplace-Beltrami operators, we instead consider L\'evy processes on a compact quantum group $\mathbb{G}$ which are invariant under the adjoint action by itself. 
This is because conjugate-invariant processes on classical compact groups have a generator given as a combination of the Laplace-Beltrami operator and the L\'evy measure \cite{Lia04}.

The generating functionals of adjoint invariant L\'evy processes on a compact quantum group $\mathbb{G}$ can be characterized as elements belonging to the center of the algebra of linear functionals on $\mathbb{G}$. 
Namely, for well-behaved compact quantum groups called Kac type, those states are known to have a one-to-one correspondence with linear functionals on the C*-algebra generated by the characters of finite dimensional irreducible unitary representations of $\mathbb{G}$ inside the universal C*-algebraic model of $\mathbb{G}$.
This leads to a concrete formula (Theorem \ref{heat semigroup}) of heat semigroups $(T_t)_t$ on ${\rm Aut}^+(B)$ with $\dim B\geq 5$ \cite{BGJ2} which implies:
\[T_t(u_{ij}^{(k)})=e^{-c_kt}u_{ij}^{(k)} \;\;{\rm where}\;c_k\sim k.\]

Up to the monoidal equivalence between ${\rm Aut}^+(B)$ and $S_n^+$ under $n=\dim B$, the heat semigroups on these quantum groups take the same form. 

Let us summarize our results.
We have ultracontractivity, hypercontractivity, the log-Sobolev inequality, and spectral gap inequality of $(T_t)_t$ with the generator $T_L$.
\begin{itemize}
\item $T_t$ is ultracontractive: $||T_tx||_\infty\leq\sqrt{f(t)}||x||_2$ where $x$ is an element of an eigenspace $V_s$ of $T_t$, $\alpha,\;\beta,\;\gamma$ are constants that only depend on $s$, and
\[f(t)=\frac{\beta^2e^{-2\alpha t}(1+e^{-2\alpha t})+2\beta\gamma e^{-2\alpha t}(1-e^{-2\alpha t})+\gamma^2(1-e^{-2\alpha t})^2}{(1-e^{-2\alpha t})^3}.\]
\item $T_t$ is hypercontractive: for each $p$ with $2<p<\infty$, there exists $\tau_p>0$ such that $||T_tx||_p\leq||x||_2$ for any $t\geq\tau_p$.
Namely the time $\tau_p$ can be estimated as 
 \[\tau_p=-\frac{n}{2}\log{Y}\]
 where $Y$ is the smallest real positive root of $\frac{Y^3-2Y^2+9Y}{(1-Y)^3}=\frac{1}{(p-1)D^2}$.
 \item$T_t$ satisfies the log-Sobolev inequality: there exists $t_0>0$ such that the following inequality holds
\[||T_t:L^2({\rm Aut}^+(B))\to L^{q(t)}({\rm Aut}^+(B))||\leq 1,\;\;\;0\leq t \leq t_0\]
where $q(t)=\frac{4}{2-t/t_0}$.

Furthermore, for $x\in L^{\infty}({\rm Aut}^+(B))_+\cap D(T_L)$, we have
\[h(x^2\log{x})-||x||_2^2\log{||x||_2}\leq-\frac{c}{2}h(xT_Lx)\]
where $c=\frac{t_0}{2}$ and $h$ is the Haar state of $\mathrm{Aut}^+(B)$.

 \item $T_t$ satisfies the spectral gap inequality: 
 \[\frac{1}{n}||x-h(x)||_2^2\leq-h(xT_Lx).\]
\end{itemize}
These follow from the same arguments as in \cite{FHLUZ}.

Moreover, by exploiting the concrete formula of heat semigroups on ${\rm Aut}^+(B)$, we show the sharpness of the Hardy-Littlewood-Sobolev inequality and the Hausdorff-Young inequality for $B$ with $\dim B\geq 5$, which are already known for $S_n^+$ in \cite{You20}.
\begin{itemize}
\item For any $p\in (1, 2]$, we have the Hardy-Littlewood-Sobolev inequality 
 \[\left(\sum_{k\geq0}\frac{n_k}{(1+k)^{s(\frac{2}{p}-1)}}||\widehat{f}(k)||_{HS}^2\right)^{\frac{1}{2}}\lesssim||f||_p\]
if and only if $s\geq 3$.
\item For any $p\in (1, 2]$, we have the Hausdorff-Young inequality 
\[\left(\sum_{k\geq0}\frac{1}{(1+k)^s}\left(n_k||\widehat{f}(k)||^2_{HS}\right)^{\frac{p'}{2}}\right)^{\frac{1}{p'}}\lesssim ||f||_p\]
for any $f\in L^p({\rm Aut}^+(B))$ if and only if $s\geq p'-2$.
\end{itemize}

In the appendix, we give another proof for the formula of heat semigroups on ${\rm Aut}^+(B)$. We thank Makoto Yamashita for conveying the idea of the proof. We show that the universal C*-algebra generated by the characters of finite dimensional irreducible unitary representations of ${\rm Aut}^+(B)$ coincides with $C([0, 1])$ by the monoidal equivalence between the representation categories of ${\rm Aut}^+(B)$ and $S_n^+$. We use the general correspondence between the central linear functions on compact quantum group $\mathbb{G}$ and the ``spherical" linear functionals on  Drinfeld double defined from $\mathbb{G}$ shown in \cite{DCFY}. 
By considering the isomorphism between the Drinfeld double of $\mathbb{G}$ and the generalized tube algebra ${\rm Tub}(\mathbb{G})$ for a compact quantum group $\mathbb{G}$ together with the strong Morita equivalence between ${\rm Tub}(\mathbb{G})$ and ${\rm Tub}({\rm Rep}(\mathbb{G}))$ \cite{NY18}, we have the correspondence of central states on ${\rm Aut}^+(B)$ and those of $S_n^+$.

\subsection*{Outline of the paper}

In Section2, we review the basics of compact quantum groups and introduce quantum automorphism groups of finite dimensional C*-algebras. 
Namely the concrete formula of heat semigroups on ${\rm Aut}^+(B)$ is included. 
In Section 3, we prove ultracontractivity, hypercontractivity, the log-Sobolev inequality, and the spectral gap inequality for the heat semigroups of ${\rm Aut}^+(B)$. 
Section 4 is devoted to the sharpness of the Sobolev embedding property of ${\rm Aut}^+(B)$. 
In the appendix, we have another proof of the concrete formula of heat semigroups on ${\rm Aut}^+(B)$.

In this paper, we write $\otimes$ for the minimal tensor product of C*-algebras. We write $\iota$ for an identity map.

\subsection*{Acknowledgements}
 
This work was supported by Japan Science and Technology Agency (JST) as part of Adopting Sustainable Partnerships for Innovative Research Ecosystem (ASPIRE), Grant Number JPMJAP2318, and Forefront Physics and Mathematics Program to Drive Transformation (FoPM), a World-leading Innovative Graduate Study (WINGS) Program, the University of Tokyo.
This is a part of the author's master's thesis, written under the supervision of Yasuyuki Kawahigashi at the University of Tokyo. The author would like to thank Professor Kawahigashi for his helpful comments. 
The author is greatly indebted to Makoto Yamashita for his useful comments in preparation of this paper and allowing the author to include the another proof for the formula of heat semigroups in the appendix.

\section{Preliminaries}
\subsection{Definition of compact quantum groups}
 For the basic theory of compact quantum groups, we follow the book \cite{NT13}.

\begin{definition}
A compact quantum group $\mathbb{G}$ consists of a unital Hopf $*$-algebra $\mathcal{O}({\mathbb{G}})$ with a coproduct $\Delta: \mathcal{O}({\mathbb{G}})\to\mathcal{O}({\mathbb{G}})\otimes\mathcal{O}({\mathbb{G}})$ together with a linear functional $h:\mathcal{O}({\mathbb{G}})\to \mathbb{C}$ called a Haar state satisfying the following conditions:
\begin{itemize}
\item $h$ is invariant in the sense that $(\iota\otimes h)\Delta(x)=h(x)1=(h\otimes\iota)\Delta(x)$ for all $x\in\mathcal{O}(\mathbb{G})$;
\item $h$ is normalized in the sense that $h(1)=1$;
\item For any $x\in \mathcal{O}(\mathbb{G}), \; h(x^*x)\geq0$.
\end{itemize}
\end{definition}

We can consider C*-completions of $\mathcal{O}(\mathbb{G})$ such as the reduced C*-algebra $C_r(\mathbb{G})$, the C*-algebra completion with respect to the GNS representation induced by the Haar state $h$.
We say that $\mathbb{G}$ is of Kac type if the Haar state is a trace.

A unitary representation of $\mathbb{G}$ on a finite dimensional Hilbert space $H$ is a unitary corepresentation of $\mathcal{O}(\mathbb{G})$ on $H$, given by a linear map $\delta: H\to H\otimes \mathcal{O}(\mathbb{G})$ satisfying $(\delta\otimes\iota)\delta=(\iota\otimes\Delta)\delta$, $(\iota\otimes\varepsilon)\delta=\iota$, and the unitary condition $(v_0, w_0)w_1^*v_1=(v, w) 1$ under the Sweedler's notation $\delta(v)=v_0\otimes v_1$ for $v, w\in H$.

A coaction of $\mathbb{G}$ on a $*$-algebra $B$ is a $*$-homomorphism $\rho:B\to B\otimes \mathcal{O}(\mathbb{G})$ such that $(\rho\otimes \iota)\rho=(\iota\otimes\Delta)\rho$ and $(\iota\otimes\varepsilon)\rho=\iota$.

\subsection{Quantum automorphism groups of finite dimensional C*-algebras}
Quantum automorphism groups of finite dimensional C*-algebras are introduced by Wang \cite[Definition 2.3]{Wan98}.

For quantum permutation groups, heat semigroups on those have ultracontractivity and hypercontractivity \cite{FHLUZ} and the sharp Sobolev embedding property follows from the concrete formula of heat semigroups \cite{You20}. 
In this paper, we show the ultracontractivity and hypercontractivity of heat semigroups and the sharpness of the Sobolev embedding property for the quantum groups ${\rm Aut}^+(B,\psi)$ for the canonical trace $\psi$ (Plancherel trace) on $B$, generalizing the case of $S_n^+$. 

For a finite dimensional C*-algebra $B$, we define the quantum automorphism group ${\rm Aut}^+(B)$ as a Hopf $*$-algebra generated by the matrix coefficients of the coaction of ${\rm Aut}^+(B)$ on $B$ that preserves a nice state called the Plancherel trace. 
If we take $B=\mathbb{C}^n$, then we have ${\rm Aut}^+(B)=S_n^+$, and if we take $B=M_n(\mathbb{C})$, we have the ``projective version" of quantum orthogonal groups for ${\rm Aut}^+(B)$.
It is known that ${\rm Aut}^+(B)$ has the rapid decay property \cite{Bra13} and the same fusion rule as $SO(3)$ \cite{Ban99}.

Consider a finite dimensional C*-algebra $B$ and a faithful state $\psi$ on $B$.
A coaction $\rho$ of a compact quantum group $\mathbb{G}$ on $B$ as above is said to be $\psi$-invariant if it satisfies $(\psi\otimes\iota)\rho(x)=\psi(x)1$ for any $x\in B$.

\begin{definition}
Let $B$ be a finite dimensional C*-algebra and $\psi$ be a faithful state on $B$.
The quantum automorphism group ${\rm Aut}^+(B, \psi)$ is a compact quantum group which is universal in the following sense: 
If there is another compact quantum group $\mathbb{G}$ which has a $\psi$-invariant coaction $\rho'$ on $B$, there is a unique $*$-homomorphism $\pi:\mathcal{O}({\rm Aut}^+(B, \psi))\to\mathcal{O}(\mathbb{G})$ such that $\rho'=(\iota\otimes\pi)\rho$.
\end{definition}

The concrete construction of ${\rm Aut}^+(B, \psi)$ in terms of generators and relations is given in \cite{Wan98}: The Hopf $*$-algebra $\mathcal{O}({\rm Aut}^+(B, \psi))$ is a universal $*$-algebra generated by the matrix coefficients $\{(\omega\otimes\iota)\rho(x):\omega\in B^*, x\in B\}$ of $\rho$, and the relations are determined by the condition that $\rho$ is a $\psi$-invariant coaction of ${\rm Aut}^+(B, \psi)$ on $B$. The Hopf $*$-algebra structure of $\mathcal{O}({\rm Aut}^+(B, \psi))$ is uniquely determined by the condition that $\rho$ is a $\psi$-invariant coaction. 

The coaction $\rho$ also defines a representation of ${\rm Aut}^+(B, \psi)$ on $L^2(B, \psi)$ \cite[Theorem 1.1]{Ban99} and it is called the fundamental representation .

A quantum automorphism group ${\rm Aut}^+(B, \psi)$ is a universal quantum analogue of the compact group ${\rm Aut}(B)$ of $*$-automorphisms on $B$. An automorphism $\alpha\in{\rm Aut}(B)$ is said to be $\psi$-preserving if $\psi\circ\alpha=\psi$. If we write ${\rm Aut}(B, \psi)$ for the subgroup of ${\rm Aut}(B)$ consisting of $\psi$-preserving automorphisms, then the algebra of coordinate functions $\mathcal{O}({\rm Aut}(B, \psi))$ is the abelianization of $\mathcal{O}({\rm Aut}^+(B, \psi))$.

Let $B$ be a finite dimensional C*-algebra equipped with a state $\psi$ on $B$ and $\delta>0$. We say $\psi$ is a $\delta$-form if the adjoint $m^*$ of the multiplication of $m:B\otimes B\to B$ with respect to the hermitian inner product defined by $\psi$ satisfies $m\circ m^*=\delta {\rm id}$. 

By \cite[Theorem 4.1]{Ban99}, we have the following fusion rules of ${\rm Aut}^+(B, \psi)$.

\begin{theorem}
The set of classes of finite dimensional irreducible representations of ${\rm Aut}^+(B, \psi)$ with a $\delta$-form $\psi$ can be labeled by the positive integers, ${\rm Irr}({\rm Aut}^+(B, \psi))=\{U_k:k\in\mathbb{N}\}$. The fundamental representation $U$ sarisfies $U\cong 1\oplus U_1$ and the fusion rule is the same as that of $SO(3)$, i.e.,
\[U_k\otimes U_s=U_{|k-s|}\oplus U_{|k-s|+1}\oplus\cdots\oplus U_{k+s-1}\oplus U_{k+s}.\]
\end{theorem}

\begin{example}
Let $B$ be a finite dimensional C*-algebra. Fix an isomorphism  $B\cong\bigoplus^m_{r=1} M_{n_r}$. We define the Plancherel trace $\psi$ of $B$ by
\[\psi(A)\coloneqq\sum^m_{r=1}\frac{n_r}{\dim(B)}{\rm Tr}_{n_r}(A_r)\]
where $A=\bigoplus^m_{r=1}A_r\in B,\; A_r\in M_{n_r}$.
This is the only tracial $\delta$-form on $B$ and we have $\delta=\sqrt{\dim B}$.
\end{example}

The Plancherel state $\psi$ is preserved by any action of ${\rm Aut}(B)$ on $B$ and we have ${\rm Aut}(B, \psi)={\rm Aut}(B)$. Therefore ${\rm Aut}^+(B, \psi)$ can be regarded as the quantum analogue of ${\rm Aut}(B)$. 
We only consider $(B, \psi)$ with the Plancherel state $\psi$ in this paper. Hence we simply write ${\rm Aut}^+(B)={\rm Aut}^+(B, \psi)$. The compact quantum group ${\rm Aut}^+(B)$ is of Kac type.

\begin{example}
Let $n\geq2$. Let $C(S_n^+)$ be the universal unital C*-algebra generated by the $n^2$ self-adjoint elements $u_{ij}$, $1\leq i, j\leq n$ satisfying the following relations:
\[u_{ij}^2=u_{ij}=u_{ij}^*\]
\[\sum_k u_{ik}=1=\sum_k u_{kj}.\]
We define the comultiplication $\Delta$ by $\Delta(u_{ij})\coloneqq\sum_k u_{ik}\otimes u_{kj}$. Then we have a compact quantum group obtained from $C(S_n^+)$ and $\Delta$ called the quantum permutation group. 

If we in addition impose commutativity to the generators, we obtain the classical permutation group. 

Note that $S_n^+$ also arises as the quantum automorphism group ${\rm Aut}^+(\mathbb{C}^n)$. The matrix $U=(u_{ij})_{i, j}$ defines the fundamental representation of $S_n^+$.
\end{example}

For $B$ with $\dim B=1, 2, 3$, it is known that ${\rm Aut}^+(B)$ coincides with $S_n$ for $n=\dim B$. If $\dim B=4$, we have $SO(3)$ or $S_4^+$ for ${\rm Aut}^+(B)$ \cite[Section 3]{Wan98}. If $n\geq 4$, $S_n^+$ has infinite dimensional function algebra and does not coincide with $S_n$ \cite[Proposition 6.2]{Wan99}. 

From \cite[Proposition 19]{DCFY} and \cite[Theorem 4.2]{DVV}, we have the following monoidal equivalence between quantum automorphism groups and $S_n^+$.

\begin{proposition}
Let $B$ be $n$-dimensional C*-algebra.
Then ${\rm Aut}^+(B)$ and $S_n^+$ are monoidally equivalent.
\end{proposition}

\subsection{Heat semigroups on ${\rm Aut}^+(B)$}

We follow the general theory of L\'evy processes on compact quantum groups in \cite{Sch93} and \cite{CFK}.

The generator of a conjugate invariant L\'evy process on classical compact connected simple Lie groups can be written as the sum of the Laplace-Beltrami operator and the integration with respect to the L\'evy measure \cite[Proposition 4.4, Proposition 4.5]{Lia04}.
Those conjugate invariant L\'evy processes are generalized to compact quantum groups cases as adjoint invariant L\'evy processes.
For ${\rm Aut}^+(B)$, we have the L\'evy-Khinchine formula for adjoint-invariant L\'evy processes and by neglecting the term corresponding to the integration with respect to the L\'evy measure, we obtain the ``Brownian motions" and heat semigroups on ${\rm Aut}^+(B)$\cite[Theorem 3.14]{BGJ2} (cf. \cite[Section 6]{Cas21}).  Namely the results for $S_n^+$ can be found in \cite{FKS}.

\begin{theorem}\label{heat semigroup}
Let $B$ be a finite dimensional C*-algebra with $\dim B\geq 4$. Then the generating functional $L$ of symmetric central quantum Markov semigroups $(T_t)_{t\geq0}$ of ${\rm Aut}^+(B)$ is of the following formula:
\[L(u_{ij}^{(k)})=\frac{\delta_{ij}}{\Pi_k(n)}\left(-a\frac{\Pi_k'(n)}{2\sqrt{n}}+\int_0^n\frac{\Pi_k(x)-\Pi_k(n)}{n-x}d\nu(x)\right)\]
where $a>0$ is a real number, $\nu$ is a finite measure on $[0,n]$, and $\Pi_k\coloneqq S_{2k}(\sqrt{x})$ for the Chebyshev polynomials of the second kind $\{S_k\}_{k=0}^\infty$ defined by the recursion 
\[S_0(x)=1,\;S_1(x)=x,\;S_1S_k=S_{k-1}+S_{k+1}.\]
\end{theorem}
The generator $T_L\coloneq(\iota\otimes L)\circ\Delta$ of the Markov semigroup $(T_t)_t$ is
\[T_L(u_{ij}^{(k)})=\frac{1}{\Pi_k(n)}\left(-a\frac{\Pi_k'(n)}{2\sqrt{n}}+\int_0^n\frac{\Pi_k(x)-\Pi_k(n)}{n-x}d\nu(x)\right)u_{ij}^{(k)}\]
and the Markov semigroup is given by $T_t={\rm exp}(tT_L)$.
In this paper, we consider the case $a=1$ and $\nu=0$, and the heat semigroups of ${\rm Aut}^+(B)$ are of the following form where $\dim B\geq 5$:
\[T_t(u_{ij}^{(k)})=e^{-c_kt}u_{ij}^{(k)} \;\;{\rm where}\;c_k\sim k.\]

\subsection{Noncommutative $L^p$ spaces}

We review the definitions of the Fourier transform on compact quantum groups \cite[Section 3]{MR3602819}. For the general theory of noncommutative $L^p$ theory, we refer to \cite{Pis03}, \cite{PX}, and for the Fourier transform on locally compact quantum groups, we follow \cite[Section 5]{MR3029493}, \cite[Section 3]{MR2669427} \cite[Section 3]{MR2747963}.

\begin{definition}
 Let $\mathbb{G}$ be a compact quantum group of Kac type and $h$ be its Haar state. 
 \begin{itemize}
 \item We define an associated von Neumann algebra $L^\infty(\mathbb{G})\coloneqq C_r(\mathbb{G})''$. We write $||\cdot||_\infty$ for the operator norm of $L^\infty(\mathbb{G})$.
 \item For any $p\in [1, \infty)$, we define the noncommutative $L^p$-space $L^p(\mathbb{G})$ as the completion of $\mathcal{O}(\mathbb{G})$ with respect to the norm $||a||_p\coloneqq (h((a^*a)^{\frac{p}{2}}))^{\frac{1}{p}}\;(a\in\mathcal{O}(\mathbb{G}))$.
 \item For any $a\in \mathcal{O}(\mathbb{G})$, we define $||a||_{HS}\coloneqq h(a^*a)^{\frac{1}{2}}$.
 \end{itemize}
 \end{definition}

\begin{definition}
Let $\mathbb{G}$ be a compact quantum group.
\begin{itemize}
 \item We define the non-commutative $\ell^{\infty}$-space by
 \[\ell^{\infty}(\widehat{\mathbb{G}})\coloneqq \oplus_{\alpha\in{\rm Irr}(\mathbb{G})} M_{n_{\alpha}}\]
 where $n_{\alpha}$ is the dimension of $U_{\alpha}$.
 \item For $1\leq p<\infty$, we define the non-commutative $\ell^p$-space by
 \[\ell^p(\mathbb{G})\coloneqq\left\{A\in \ell^{\infty}(\mathbb{G}):\sum_{\alpha\in{\rm Irr}(\mathbb{G})}n_{\alpha}{\rm tr}_{M_{n_{\alpha}}}(|A_{\alpha}|^p)<\infty\right\}.\]
 where ${\rm tr}_{M_{n_{\alpha}}}$ is the usual trace on $M_{n_{\alpha}}$.
 \end{itemize}
 \end{definition}
 
 \begin{definition}
Let $\mathbb{G}$ be a compact quantum group.
We define the Fourier transform $\mathcal{F}:L^1(\mathbb{G})\to l^{\infty}(\widehat{\mathbb{G}}),\;\phi\mapsto\widehat{\phi}=(\widehat{\phi}(\alpha))_{\alpha\in{\rm Irr}(\mathbb{G})}$ by
\[\widehat{\phi}(\alpha)_{ij}=\phi((u_{ij}^{\alpha})^*)\]
for all $1\leq i,j\leq n_\alpha$ under the identification $L^1(\mathbb{G})=L^{\infty}(\mathbb{G})_*$.
\end{definition}

If $\mathbb{G}$ is of Kac type, we call 
$\sum_{\alpha\in{\rm Irr}(\mathbb{G})} n_{\alpha}{\rm tr} (\widehat{\phi}(\alpha)u^{\alpha})=\sum_{\alpha\in{\rm Irr}(\mathbb{G})} \sum_{i, j=1}^{n_{\alpha}}n_{\alpha}\widehat{\phi}(\alpha)_{ij}u^{\alpha}_{ij}$
 the Fourier series of $\phi\in L^1(\mathbb{G})$ and denote it by 
 $\phi\sim \sum_{\alpha\in{\rm Irr}(\mathbb{G})} n_{\alpha}{\rm tr} (\widehat{\phi}(\alpha)u^{\alpha})$.
 If $f\in \mathcal{O}(\mathbb{G})$, we indeed have $f= \sum_{\alpha\in{\rm Irr}(\mathbb{G})} n_{\alpha}{\rm tr} (\widehat{f}(\alpha)u^{\alpha})$ because $\widehat{f}(\alpha)=0$ for all but finitely many $\alpha$. 

By the Plancherel theorem and the complex interpolation theorem, $\mathcal{F}$ can be regarded as a contractive map from $L^p(\mathbb{G})$ into $\ell^{p'}(\mathbb{G})$ for $1\leq p\leq2$ where $p'$ is the conjugate of $p$. 

Next we review some facts on complex interpolation methods. For details, we refer to \cite[Section 1]{Xu96}.
\begin{definition}
Let $\{E_k\}_{k\in \mathbb{Z}}$ be a family of Banach spaces and $\mu$ be a positive measure on $\mathbb{Z}$. We define vecor valued $\ell^p$-space by
\[\ell^p(\left \{E_k\right\}_{k\in\mathbb{Z}},\mu)=\left \{(x_k)_{k\in \mathbb{Z}}:x_k\in E_k~\text{for~all~}k\in \mathbb{Z}~\mathrm{and~}\left (||x_k||_{E_k}\right )_{k\in \mathbb{Z}} \in \ell^p(\mathbb{Z}, \mu)\right\}\]
where the norm structure is
\[||(x_k)_{k\in  \mathbb{Z}}||_{\ell^p(\left \{E_k\right\}_{k\in  \mathbb{Z}},\mu)}= \left \{ \begin{array}{cc}\left ( \sum_{k\in  \mathbb{Z}}||x_k||_{E_k}^p \mu(k) \right )^{\frac{1}{p}},&~\mathrm{if~}1\leq p<\infty,\\ \sup_{k\in  \mathbb{Z}}\left \{ ||x_k||_{E_k}\right \},&~\mathrm{if~}p=\infty. \end{array}\right.\]
\end{definition}

If $(E_k,F_k)$ is a compatible pair of Banach spaces for any $k\in \mathbb{Z}$ and $\mu_0,\mu_1$ are two positive measures on $\mathbb{Z}$, then for any $\theta\in (0, 1)$ we have
\[\left (\ell^{p_0}(\left \{E_k\right\}_{k\in \mathbb{Z}},\mu_0),\ell^{p_1}(\left \{F_k\right\}_{k\in \mathbb{Z}},\mu_1) \right )_{\theta}= \ell^{p}(\left \{(E_k,F_k)_{\theta}\right\}_{k\in \mathbb{Z}},\mu)\]
isometrically, where $\displaystyle \frac{1-\theta}{p_0}+\frac{\theta}{p_1}=\frac{1}{p}$ and $\mu=\displaystyle \mu_0^{\frac{p(1-\theta)}{p_0}}\mu_1^{\frac{p \theta}{p_1}}$.

\section{Ultracontracitivity and hypercontractivity of heat semigroups on ${\rm Aut}^+(B)$}
 
We obtain the ultracontractivity and hypercontractivity of the heat semigroup on ${\rm Aut}^+(B)$ as those of $S_n^+$ investigated in \cite{FHLUZ}. In the following, we write $\dim B=n$. We follow the same outline as in the case for $S_n^+$ \cite{FHLUZ}, and we sketch the necessary modifications.

First we recall by Theorem \ref{heat semigroup} that the eigenvalues $\lambda_k$ of the generator $T_L$ of a heat semigroup are given by 
\[\lambda_k=-\frac{\Pi_k'(n)}{2\sqrt{n}\Pi_k(n)}\]
with multiplicities $m_k=\Pi_k(n)^2$ \cite[Section 1.4]{FHLUZ} (cf. \cite[Remark 10.4]{CFK}). 

Consider the zeros of the Chevyshev polynomial of the second kind $S_k$:
\[S_k(x)=(x-x_1)\cdots(x-x_k).\]
Then we have the following for $n \geq 5$ as in \cite[Lemma 1.8]{FHLUZ}:
\[\frac{k}{n}\leq -\lambda_k=\frac{\Pi_k'(n)}{2\sqrt{n}\Pi_k(n)}=\frac{1}{2\sqrt{n}}\sum_{s=1}^n\frac{1}{\sqrt{n}-x_s}\leq \frac{k}{\sqrt{n}(\sqrt{n}-2)}.\]

\subsection{Ultracontractivity}

 \begin{definition}
 Let $\mathbb{G}$ be a compact quantum group of Kac type.
 A semigroup $\{T_t\}$ of $T_t:L^2(\mathbb{G})\to L^\infty(\mathbb{G})$ is said to have ultracontractivity if $T_t$ is bounded for any $t$.
 \end{definition}

The following theorem \cite[Theorem 2.1]{FHLUZ} gives a condition of a semigroup on a compact quantum group of Kac type to be ultracontractive as follows:

\begin{theorem}
Let $\{T_t\}$ be a heat semigroup on a compact quantum group $\mathbb{G}$ of Kac type and ${\rm Irr}(\mathbb{G})$ is labelled by integers $s\geq 0$, so that ${\rm Irr}(\mathbb{G})=\{U_s|s=0, 1, 2...\}$. Assume that $\{T_t\}$ satisfies the following conditions:
\begin{itemize}
\item The subspaces $V_s$ spanned by the matrix coefficients of $U_s\in{\rm Irr}(\mathbb{G})$ are eigenspaces for the generator $T_L$ of $\{T_t\}$, i.e., $T_L x=\lambda_s x$ for $x\in V_s$.
\item We have an estimate of the eigenvalues $\lambda_s$ of the form $\lambda_s\leq-\alpha s$ for some $\alpha>0$. 
\item We have an inequality of the form 
\[||x||_\infty\leq(\beta s+\gamma)||x||_2\]
for $x\in V_s$ where $\beta, \gamma\geq 0$ are independent of $s$.
\end{itemize}
Then $T_t$ is ultracontractive: $||T_tx||_\infty\leq\sqrt{f(t)}||x||_2$ where
\[f(t)=\frac{\beta^2e^{-2\alpha t}(1+e^{-2\alpha t})+2\beta\gamma e^{-2\alpha t}(1-e^{-2\alpha t})+\gamma^2(1-e^{-2\alpha t})^2}{(1-e^{-2\alpha t})^3}.\]

\end{theorem}

The following theorem \cite[Theorem 4.10]{Bra13} implies that the theorem above can be applied to a heat semigroup on ${\rm Aut}^+(B)$ with $\dim B\geq5$ by taking $\alpha=\frac{1}{n},\;\beta=2D,\;\gamma=D$ as the same for $S_n^+$.

\begin{theorem}
Let $B$ be a finite dimensional C*-algebra with $\dim B\geq5$. Then there exists a constant $D>0$ depending only on $\dim B$ such that 
\[||x||_{\infty}\leq D(2k+1)||x||_2\;\;(k\in \mathbb{N},\;x\in V_k)\]
where $V_k$ is the linear span of the matrix coefficients of $U_k\in{\rm Irr}({\rm Aut}^+(B))$.
\end{theorem}

Let $\mathbb{G}$ be a compact matrix quantum group. Its discrete dual is said to have the rapid decay property with $r_k\lesssim (1+k)^\beta$ if there exists $C$ and $\beta$ such that 
\[||x||_{\infty}\leq C(1+k)^{\beta}||x||_2\]
for any $x\in V_k$. 
The theorem above shows that ${\rm Aut}^+(B)$ has the rapid decay property. 

\begin{remark}
For ${\rm Aut}^+(B)$ with $\dim B=4$, we just have $S_4^+$ or $SO(3)$, and ultracontracitivity for $S_4^+$ is obtained by concrete calculation in \cite[Section 2.2]{FHLUZ}. In this case, eigenvalues of heat semigroups satisfy
\[\lambda_k=-\frac{k(k+2)}{6}.\]
\end{remark}

\subsection{Hypercontractivity and the log-Sobolev inequality}

\begin{definition}
We say that a semigroup $\{T_t\}$ is hypercontractive if for each $p$ with $2<p<\infty$, there exists $\tau_p>0$ such that $||T_tx||_p\leq||x||_2$ for any $t\geq\tau_p$.
\end{definition}

Note that if a semigroup $\{T_t\}$ is hypercontractive, we also have $||T_tx||_p\leq||x||_2$ for $1\leq p \leq 2$.

We have the hypercontractivity of ${\rm Aut}^+(B)$ as that for $S_n^+$ proved in \cite[Theorem 2.4]{FHLUZ}. We also have the hypercontractivity estimate as in \cite{FHLUZ} because the proofs only use the form of the heat semigroups and do not depend on the other properties of $S_n^+$ (cf. \cite[Proposition 2.5, Theorem 2.6]{FHLUZ}).

\begin{theorem}\label{hypercontractivity}
Let $B$ be a finite dimensional C*-algebra.
Then the heat semigroup $\{T_t\}$ of ${\rm Aut}^+(B)$ is hypercontractive.

Furthermore, the hypercontractivity of $\{T_t\}$ is achieved at least from the time $\tau_p$ given by
 \[\tau_p=-\frac{n}{2}\log{Y}\]
 where $Y$ is the smallest real positive root of $\frac{Y^3-2Y^2+9Y}{(1-Y)^3}=\frac{1}{(p-1)D^2}$.
\end{theorem}

\begin{proof}
By \cite[Theorem 1]{RX}, as for $S_n^+$, we have
\[||x||_p^2\leq||h(x)1||_p^2+(p-1)||x-h(x)1||_p^2,\;\;\;\;\;\;\;\;x\in L^{\infty}({\rm Aut}^+(B))\]
for $2<p<\infty$ by considering $L^{\infty}({\rm Aut}^+(B))$ and the Haar state. Therefore by writing $x=h(x)1+\sum_{k\geq1} x_k$ for $x\in \mathcal{O}({\rm Aut}^+(B))$ where $x_k\in V_k$, we have the following as in \cite[Theorem 2.4]{FHLUZ}:
\[||T_t(x)||_p^2\leq||T_t(h(x)1)||_p^2+(p-1)||T_t(x-h(x)1)||_p^2\]
\[\leq|h(x)|^2+(p-1)\left(\sum_{k\geq1}||T_t(x_k)||_p\right)^2
\leq|h(x)|^2+(p-1)\left(\sum_{k\geq1}e^{\lambda_k t}||x_k||_p\right)^2\]
\[\leq|h(x)|^2+(p-1)\left(\sum_{k\geq1}e^{\lambda_k t}||x_k||_{\infty}\right)^2
\leq|h(x)|^2+(p-1)\left(\sum_{k\geq1}e^{\lambda_k t}(\beta k+\gamma)||x_k||_2\right)^2\]
\[\leq |h(x)|^2+(p-1) \sum_{k\geq1}(e^{\lambda_k t}(\beta k+\gamma))^2\sum _{k\geq1} ||x_k||_2^2 \leq||x||_2^2\]
for $t\geq \tau_p$ with $\tau_p$ such that
\[(p-1) \sum_{k\geq1}(e^{\lambda_k \tau_p}(\beta k+\gamma))^2\leq 1.\]

 Let us consider the estimation of $\tau_p$. Using $\lambda_k \le -k/n$, we get 
 \[(p-1) \sum_{k\geq1}(e^{\lambda_k t}(\beta k+\gamma))^2 \le (p-1) \sum_{k\geq1}(e^{-\frac{k}{n} t}(\beta k+\gamma))^2.\]
 If we look at the condition $(p-1) \sum_{k\geq1}(e^{-\frac{k}{n} t}(\beta k+\gamma))^2= 1$, this is equivalent to 
 \[\frac{Y_t^3-2Y_t^2+9Y_t}{(1-Y_t)^3}=\frac{1}{(p-1)D^2}\]
  for $Y_t = exp(-2t/n)$. We can then take $\tau_p$ as the smallest root of this equation because $Y_t$ is a decreasing function in $t
 $.

\end{proof}

As hypercontractivity is equivalent to the logarithmic Sobolev inequalities, we also have the following (cf. \cite[Proposition 3.4, Theorem 3.5]{FHLUZ}, \cite{Gro75}, \cite[Section 3]{OZ}).

\begin{proposition}
There exists $t_0>0$ such that the following inequality holds for the heat semigroup $\{T_t\}$ of ${\rm Aut}^+(B)$ with $\dim B\geq5$:
\[||T_t:L^2({\rm Aut}^+(B))\to L^{q(t)}({\rm Aut}^+(B))||\leq 1,\;\;\;0\leq t \leq t_0\]
where $q(t)=\frac{4}{2-t/t_0}$.

Furthermore, for $x\in L^{\infty}({\rm Aut}^+(B))_+\cap D(T_L)$, we have
\[h(x^2\log{x})-||x||_2^2\log{||x||_2}\leq-\frac{c}{2}h(xT_Lx)\]
where $c=\frac{t_0}{2}$.
\end{proposition}

 \begin{proposition}
 For a heat semigroup $\{T_t\}$ of ${\rm Aut}^+(B)$ with $\dim B\geq5$, there exists $\varepsilon_0\geq0$ such that for any $p\geq4-\varepsilon_0$,
 \[||T_t||_{2\to p}\leq1, \text{for\;all}\;t\geq\frac{dn}{2}\log{(p-1)}+(1-\frac{2}{p}n\log D)\]
 with $d=\frac{\log(11+\sqrt{105})-\log2}{\log3}.$
 \end{proposition}
 
 \begin{proof}
 By the H\"older inequality, for $p\geq1$, we have
 \[||x_k||_p\leq(D(k+1))^{1-\frac{2}{p}}||x_k||_2, \;x_k\in V_k.\]
 Therefore
 \[||T_t(x)||_p^2\leq |h(x)|^2+(p-1)\left(\sum_{k\geq1} e^{\lambda_k t}||x_k||_p\right)^2\]
 \[\leq |h(x)|^2+(p-1)\left(\sum_{k\geq1} e^{\lambda_k t}(D(2k+1))^{1-\frac{2}{p}}||x_k||_2\right)^2\]
 \[\leq  |h(x)|^2+(p-1) \sum_{k\geq1} e^{2 \lambda_k t}(D(2k+1))^{2(1-\frac{2}{p})}||x_k||_2^2.\]
 When $t\geq \frac{dn}{2}\log{(p-1)}+(1-\frac{2}{p})n\log{D}$ and $s\geq1$, 
 \[2\lambda_kt\leq-dk\log{(p-1)}-2\left(1-\frac{2}{p}\right)k\log{D}\]
 \[\leq -dk\log{(p-1)}-2\left(1-\frac{2}{p}\right)\log{D}\]
 and $e^{2\lambda_kt}\leq (p-1)^{-dk}D^{-2(1-\frac{2}{p})}.$
 
Consider $\phi(p)\coloneqq (p-1)^{1-dk}(2k+1)^{2(1-\frac{2}{p})}$.
 Therefore it suffices to show that for some $\varepsilon_0$, for any $p\geq 4-\varepsilon_0$,
 \[R_p\coloneqq\sum_{k\geq1}\phi(p)=\sum_{k\geq1}(p-1)^{1-dk}(2k+1)^{2(1-\frac{2}{p})}\leq 1.\]
Note that 
 \[R_4=\sum_{k\geq1}\frac{2k+1}{3^{dk-1}}=\frac{3(3\cdot3^d-1)}{(3^d-1)^2}\]
 and $d$ needs to satisfy $d\geq \frac{\log(11+\sqrt{105})-\log2}{\log3}$.
 
We have that $\phi'(p)\leq0$ if and only if
\[\frac{4(p-1)}{p^2}\leq\frac{dk-1}{\log{(2k+1)}}.\]
As $f_1(p)=\frac{4(p-1)}{p^2}$ 
is decreasing for $p\geq2$ and $f_2(k)=\frac{dk-1}{\log{(2k+1)}}$
 is increasing for 
 $k\geq1$, $d=\frac{\log(11+\sqrt{105})-\log2}{\log3}>\log{3}+1$ satisfies the following:
\[f_1(p)\leq f_1(2)=1<\frac{d-1}{\log{3}}=f_2(1)\leq f_2(s)\]
for any $p\geq2, s\geq1$.
Now we have that $d =\frac{\log(11+\sqrt{105})-\log2}{\log3}$ satisfies the desired condition.
 \end{proof}

\subsection{Spectral gap inequality}
 
 As in \cite[Sectoin 3]{FHLUZ}, we have the spectral gap inequality of the heat semigroup of ${\rm Aut}^+(B)$. 
 
  \begin{definition}
 Let $\mathbb{G}$ be a compact quantum group and $\{T_t\}$ be a semigroup on $\mathbb{G}$ with the generator $T_L$. We say that $\{T_t\}$ verifies a spectral gap inequality with constant $m>0$ if the following inequality is satisfied for any $x\in \mathcal{O}(\mathbb{G})_+$:
 \[m||x-h(x)||_2^2\leq-h(xT_Lx).\]
 \end{definition}
 
  \begin{proposition}
 The heat semigroup $\{T_t\}$ of ${\rm Aut}^+(B)$ with $\dim B\geq5$ verifies the spectral gap inequality with constant $m=\frac{1}{n}$ for any $x\in \mathcal{O}(\mathbb{G})_+$.
 \end{proposition}
 
 The proof is the same as that of \cite[Proposition 3.2]{FHLUZ}.
 \begin{proof}
 For $x\in \mathcal{O}(\mathbb{G})$, we write $x=\sum_k x_k$ where $x_k\in V_k$ are the elements of the eigenspaces of $T_L$. Then we have 
 \[h(xT_Lx)=\sum_k -\frac{\Pi_k'(n)}{2\sqrt{n}\Pi_k(n)}||x_k||_2^2.\]
 As $V_k$ are in orthogonal direct sum and $\frac{k}{n}\leq -\lambda_k$, we have that
 \[-h(xT_Lx)\geq \frac{1}{n}||x||_2^2.\]
 We also have $||x-h(x)||_2\leq||x||_2$ because $V_0=\mathbb{C}1$ and hence $||x-h(x)||_2^2\leq-nh(xT_Lx).$
 \end{proof}
 
 \section{The sharp Sobolev embedding properties for ${\rm Aut}^+(B)$}

Another application of the Theorem \ref{heat semigroup} is the sharp Sobolev embedding property. It is proved for $S_n^+$ in \cite{You20}. In this section, we see that the sharp Sobolev embedding property can also be obtained for ${\rm Aut}^+(B)$ in the same way. Although the proofs are the same for $S_n^+$, we include the proofs for ${\rm Aut}^+(B)$ for the sake of the completeness.

In the following, we use the concept of length function for compact matrix quantum groups, which include ${\rm Aut}^+(B)$.

 \begin{definition}
A compact quantum group $\mathbb{G}$ is called a compact matrix quantum group if there is a unitary representation $U$ such that any $U_{\alpha}\in {\rm Irr}(\mathbb{G})$ is a irreducible component of $U^{\otimes n}$ for some $n\in \mathbb{N}\cup\{0\}$. 
\end{definition}

\begin{definition}
Let $\mathbb{G}$ be a compact matrix quantum group with a finite dimensional unitary representation $U$ satisfying the above condition.
\begin{itemize}
\item We define a length function on ${\rm Irr}(\mathbb{G})$ with respect to $U$ by
\[|\alpha|=\min\{n\in\{0\}\cup\mathbb{N}:U_{\alpha}\;{\rm is\;an\;irreducible\;component\;of}\;U^{\otimes n}\}\] 
for $\alpha\in{\rm Irr}(\mathbb{G})$.
\item We define the $k$-sphere $S_k$ by
\[S_k\coloneqq\{\alpha\in{\rm Irr}(\mathbb{G}):|\alpha|=k\}.\]
\end{itemize}
\end{definition}

Note that for ${\rm Aut}^+(B)$, by considering the length function with respect to the fundamental representation corresponding to the defining coaction, we have $|U_k|=k$ for $U_k\in {\rm Irr}({\rm Aut}^+(B))$.

\subsection{The sharp Sobolev embedding properties}

It is known that we have the Hardy-Littlewood-Sobolev inequality for ${\rm Aut}^+(B)$ as follows \cite[Theorem 4.5]{You20}. In this section, we prove its sharpness.

\begin{theorem}\label{Sobolev}
Let $B$ be a finite dimensional C*-algebra with $\dim B\geq5$.
Then for any $p\in (1, 2]$ and any $f\in L^p({\rm Aut}^+(B))$, we have
 \[\left(\sum_{k\geq0}\frac{n_k}{(1+k)^{s(\frac{2}{p}-1)}}||\widehat{f}(k)||_{HS}^2\right)^{\frac{1}{2}}\lesssim||f||_p\]
if $s\geq 3$.
\end{theorem}

First we review that the following \cite[Proposition 5.2]{You20} is applicable to ${\rm Aut}^+(B)$.

\begin{proposition}\label{infty<2}
Let $\mathbb{G}$ be a compact matrix quantum group whose dual has the rapid decay property with $r_k\lesssim (1+k)^\beta$ and let $\omega:\{0\}\cup\mathbb{N}\to (0, \infty)$ be a positive function such that $C_{\omega}=\sum_{k\geq0}\frac{(1+k)^{2\beta}}{e^{2\omega(k)}}<\infty$.
Then we have 
\[\left \lVert \sum_{\alpha\in {\rm Irr}(\mathbb{G})} \frac{n_\alpha}{e^{\omega(|\alpha|)}} {\rm tr}(\hat{f}(\alpha)u^{\alpha})\right\rVert_{\infty} \lesssim \sqrt{C_{\omega}}||f||_2\]
for any $f\in L^2(\mathbb{G})$.
In particular, we have
\[\left\lVert \sum_{\alpha\in {\rm Irr}(\mathbb{G}) }\frac{n_\alpha}{(1+|\alpha|)^s} {\rm tr}(\hat{f}(\alpha)u^{\alpha})\right\rVert_{\infty} \lesssim ||f||_2\]
for any $f\in L^2(\mathbb{G})$.
\end{proposition}

We also have a version of \cite[Theorem 5.3]{You20} for ${\rm Aut}^+(B)$. 
 
\begin{theorem}\label{K}
There exists a universal constant $K>0$ such that $\sum_{k\geq0}\frac{(1+k)^2}{e^{2\omega(k)}}\leq KC^2$ for a finite dimensional C*-algebra $B$ with $\dim B\geq4$ and a positive function $\omega:\{0\}\cup\mathbb{N}\to(0,\infty)$ satisfying 
\[\left\lVert\sum_{\alpha\in{\rm Irr}(\mathbb{G})}\frac{n_{\alpha}}{e^{\omega(|\alpha|)}}{\rm tr}(\widehat{f}(\alpha)u_{\alpha})\right\rVert_{\infty}\leq C||f||_2\]
for any $f\in L^2({\rm Aut}^+(B))$.
\end{theorem}

The proof of \cite[Theorem 5.3]{You20} is also valid for ${\rm Aut}^+(B)$ because we have the following statement by \cite[Remark 4.6, Lemma 4.7]{You18}.

\begin{lemma}\label{Phi}
Let $B$ be a finite dimensional C*-algebra with $\dim B\geq4$.
Consider the characters of irreducible representations $\chi_n$ of ${\rm Aut}^+(B)$ and $\widetilde{\chi_n}$ of $SO(3)$.
For $f\sim \sum_{n\geq0}c_n\chi_n \in L^p({\rm Aut}^+(B))$, the associate function $\Phi(f)\sum_{n\geq0}c_n\widetilde{\chi_n} \in L^p(SO(3))$ has the same norm:
\[||f||_{L^p({\rm Aut}^+(B))}=||\Phi(f)||_{L^p(SO(3))}\]
for any $p\in [1, \infty]$.
\end{lemma}

\begin{proof}[Proof of Theorem \ref{K}]
By Lemma \ref{Phi} and the fact that there is the Poisson semigroup $(\mu_t)$ on $SO(3)$ satisfying $\mu_t\sim \sum_{k\geq0} e^{-t\kappa_k^{\frac{1}{2}}}(2k+1)\widetilde{\chi_k}$, we have that 
\[\sum_{k\geq0} e^{-t\kappa_k^{\frac{1}{2}}}e^{-2\omega(k)}(1+k)^2\sim \left\lVert \sum_{k\geq0}e^{-t\kappa_k^{\frac{1}{2}}}e^{-2\omega(k)}(1+2k)\widetilde{\chi_k}\right\rVert_{L^{\infty}(SO(3))}\]
\[\leq\left\lVert \sum_{k\geq0}e^{-t\kappa_k^{\frac{1}{2}}}e^{-2\omega(k)}(1+2k){\chi_k}\right\rVert_{L^{\infty}({\rm Aut}^+(B))}\leq C\left\lVert \sum_{k\geq0}e^{-t\kappa_k^{\frac{1}{2}}}e^{-\omega(k)}(1+2k){\chi_k}\right\rVert_{L^{2}({\rm Aut}^+(B))}\]
\[\leq C^2 ||\widetilde{\mu_t}||_{L^1({\rm Aut}^+(B))}=C^2||\mu_t||_{L^1(SO(3))}=C^2.\]
By taking the limit $t\to 0^+$, we get $\sum_{k\geq0}e^{-2\omega(k)}(1+k)^2\leq KC^2$ for a universal constant $K>0$.
\end{proof}

\begin{remark}
Let $\mathbb{G}$ be a compact matrix quantum group of Kac type. The degree of the rapid decay property for the dual of $\mathbb{G}$ is the infimum of positive numbers $s\geq0$ such that 
\[||f||_{\infty}\lesssim \left(\sum_{\alpha\in {\rm Irr}(\mathbb{G})}(1+|\alpha|)^{2s}n_{\alpha}||\widehat{f}(\alpha)||^2_{HS}\right)^{\frac{1}{2}}\]
for any $f \in \mathcal{O}(\mathbb{G})$.
The rapid decay degree of the dual of ${\rm Aut}^+(B)$ is $\frac{3}{2}$. This can be deduced from the Proposition \ref{infty<2} and Theorem \ref{K} as in \cite[Cororally 5.4]{You20}.
\end{remark}

Consider the semigroup $\{S_t\}$ defined by $S_t\coloneqq e^{-t}T_t$ for a heat semigroup $\{T_t\}$ of ${\rm Aut}^+(B)$.
By showing the ultracontractivity of this semigroup, we obtain the sharp Sobolev embedding property for ${\rm Aut}^+(B)$ as that of $S_n^+$ investigated in \cite{You20}.

\begin{proposition}\label{iff3}
Let $B$ be a finite dimensional C*-algebra with $\dim B\geq5$.
Then there exists a universal constant $K>0$ such that
\[||S_t(f)||_{\infty}\leq\frac{K||f||_2}{t^{\frac{s}{2}}}\;\;\text{for\;all}\;f\in L^2({\rm Aut}^+(B)) \;and\;t>0\]
if and only if $s\geq3$.
\end{proposition}

Similar to Theorem \ref{hypercontractivity}, we can follow the case of $S_n^+$ \cite[Cororally 6.2]{You20}. Let us indicate the key steps.

\begin{proof}
The image of heat semigroups can be written as $T_t(u_{ij}^{(k)})=e^{-tc_k}u_{ij}^{(k)}$ where $c_k\sim k$. 

By \cite[Lemma 6.1 (2)]{You20}, $\displaystyle \sup_{0<t<\infty}\left \{ t^s \sum_{k\geq 0}\frac{(1+k)^2}{e^{2t(1+c_k)}} \right\}<\infty$ holds if and only if $s\geq 3$.
Therefore if we assume $s\geq 3$, we have
\[C_{\omega}=\sum_{k\geq 0}\frac{(1+k)^2}{e^{2t(1+c_k)}}\lesssim \frac{1}{t^s}.\]
Proposition \ref{infty<2} is applicable for $\omega:k\mapsto t(1+c_k)$ and we have
\[||S_t(f)||_{\infty}= ||\sum_k \frac{n_k}{e^{t(1+c_k)}}{\rm tr} (\widehat{f}(\alpha)u^{\alpha})||_{\infty}\lesssim \frac{||f||_2}{t^{\frac{s}{2}}}.\]

Conversely, if we have $||S_t(f)||_{\infty}\leq\frac{K||f||_2}{t^{\frac{s}{2}}}$ for all $f\in L^2({\rm Aut}^+(B))$, then we have 
\[\sum_{k\geq 0}\frac{(1+k)^2}{e^{2t(1+c_k)}}\lesssim \frac{1}{t^s}\]
by Theorem \ref{K} and we have $s\geq3$.
\end{proof}

Now we can apply the following theorem \cite[Theorem 1.1]{Xio17} (cf. \cite[Theorem 3.1]{You20})).

\begin{theorem}\label{Xiong}
Let $\mathbb{G}$ be a compact quantum group of Kac type and $\{T_t\}$ be a heat semigroup on $L^{\infty}(\mathbb{G})$ with the generator $T_L$. Then for the semigroup $\{S_t\}\coloneqq \{e^{-t}T_t\}$ and $s>0$, the following are equivalent:\\
{\rm (1)} There exist $p$ and $q$ with $1\leq p<q\leq\infty$ such that
\[||S_t(x)||_{\infty}\lesssim \frac{||x||_p}{t^{s(\frac{1}{p}-\frac{1}{q})}}\;\;\text{for\;all}\;x\in L^p(\mathbb{G})\;and\;t>0.\]
{\rm (2)} For any $1<p<q<\infty$, 
\[||(1-T_L)^{-s(\frac{1}{p}-\frac{1}{q})}(x)||_q\lesssim||x||_p\;\;\text{for\;all}\;x\in L^p(\mathbb{G}).\]

\end{theorem}

Now we have the sharpness of the Hardy-Littlewood-Sobolev inequality for ${\rm Aut}^+(B)$.
Proposition \ref{iff3} implies that Theorem \ref{Xiong} (1) is satisfied for $s\geq3,\;p=2,\;q=\infty$.
Therefore we obtain the following inequality from Theorem \ref{Xiong} (2) for $s=3,\;q=2$  (cf. \cite[Example 4 (2)]{You20}).

 \[\left(\sum_{k\geq0}\frac{n_k}{(1+k)^{3(\frac{2}{p}-1)}}||\widehat{f}(k)||_{HS}^2\right)^{\frac{1}{2}}\lesssim||f||_p\]
 for any $p\in(1,2]$ and $f\in L^p({\rm Aut}^+(B))$ with $\dim B\geq5$.

\subsection{The sharp Hausdorff-Young inequality}
  The Hausdorff-Young inequalities for general compact quantum groups state that the Fourier transform $\mathcal{F}:L^p(\mathbb{G})\to\ell^{p'}(\widehat{\mathbb{G}})$ is contractive for any $1\leq p\leq2$ where $1<p'$ is the conjugate index of $p$. 

We have the Hausdorff-Young inequality for ${\rm Aut}^+(B)$ by \cite[Theorem 4.1]{You20} as follows.

\begin{theorem}
Let $B$ be a finite dimensional C*-algebra with $\dim B\geq5$.
Then for any $p\in (1, 2]$, we have
\[\left(\sum_{k\geq0}\frac{1}{(1+k)^{ (p'-2)}}\left(\sum_k n_k|\widehat{f}(\alpha)||^2_{HS}\right)^{\frac{p'}{2}}\right)^{\frac{1}{p'}}\lesssim||f||_p\]
for any $f\in L^p({\rm Aut}^+(B))$.
\end{theorem} 

We also have the sharpness of the Hausdorff-Young inequality. This can be shown in the same way as for $S_n^+$ \cite[Corollary 6.3]{You20}.

\begin{theorem}
 Let $1<p\leq2$ and $B$ be a finite dimensional C*-algebra with $\dim B\geq5$.
 \[\left(\sum_{k\geq0}\frac{1}{(1+k)^s}\left(n_k||\widehat{f}(k)||^2_{HS}\right)^{\frac{p'}{2}}\right)^{\frac{1}{p'}}\lesssim ||f||_p\]
 for any $f\in L^p({\rm Aut}^+(B))$ if and only if $s\geq p'-2$.
 \end{theorem}
 
 \begin{proof}
 It is enough to prove the ``only if" part.
 By \cite[Corollary 3.9]{You18}, we have
 \[\left(\sum_{k\geq0}\frac{1}{(1+k)^{4-2p}}\left(n_k ||\widehat{f}(k)||_{HS}^2\right)^{\frac{p}{2}}\right)^{\frac{1}{p}}\lesssim ||f||_p.\]
 Therefore 
 \[||\mathcal{F}||_{L^p({\rm Aut}^+(B))\rightarrow \ell^p(\left \{V_k\right\}_{k\geq 0},\mu_0)},~||\mathcal{F}||_{L^p({\rm Aut}^+(B))\rightarrow \ell^{p'}(\left \{V_k\right\}_{k\geq 0},\mu_1)}<\infty \] 
where $\mu_0(k)=(1+k)^{2p-4}$ and $\mu_1(k)=(1+k)^{-s}$. By considering complex interpolation, we have 
\[||\mathcal{F}||_{L^p({\rm Aut}^+(B))\rightarrow \ell^{2}(\left \{V_k\right\}_{k\geq 0},\mu)}<\infty~\mathrm{with}~\mu(k)=(1+k)^{2-\frac{4}{p}-\frac{s}{p'}}\] and the consequence of the previous section implies that $\displaystyle -2+\frac{4}{p}+\frac{s}{p'}\geq 3(\frac{2}{p}-1)$ which is equivalent to $s\geq p'-2$.
 \end{proof}

 \section{Appendix: another proof for the formula of heat semigroups on ${\rm Aut}^+(B)$}
 
The author thanks M. Yamashita for the ideas in this appendix.
 
We are going to prove that ad-invariant L\'{e}vy processes on ${\rm Aut}^+(B)$ has the same form as those of $S_n^+$ for $n=\dim(B)$.

\begin{theorem}
The ad-invariant generating functionals on $\mathcal{O}({\rm Aut}^+(B))$ with $n=\dim(B)$ are of the form
\[
\hat{L}=L\circ \widetilde{{\rm ad}}_h
\]
with $L$ defined on $C^1([0, n])\subset C^*(\chi_k : k\in \mathbb{N})\cong C([0,n])$ by
\[
Lf = -a f'(n) + \int_{0}^n \frac{f(x) - f(n)}{n-x} {\rm d}\nu(x)
\]
where $a>0$ is a real number and $\nu$ is a finite measure on $[0,n]$ satisfying $\nu(\{n\})=0$. Furthermore, $a$ and $\nu$ are uniquely determined by $L$.

\end{theorem}

The key to the proof of the theorem is the following proposition:

 \begin{proposition}
 For ${\rm Aut}^+(B)$ with $n=\dim(B)$, 
 \[C^*(\chi_k : k\in \mathbb{N})\cong C([0,n]).\]
 \end{proposition}
 
 This is obtained by considering the slice maps of central functions on ${\rm Aut}^+(B)$ in \cite[Theorem 3.14]{BGJ2} along the line of \cite[Section 2]{DCFY} and \cite[Proposition 6.3]{MR3084500}. 
 However, in this paper, we give another proof by considering the correspondence of central functionals on compact quantum groups and the ``spherical" states on its Drinfeld double. The monoidal equivalence of ${\rm Aut}^+(B)$ and $S_n^+$ gives the coincidence of the Drinfeld doubles by considering their tube algebras.

Let $\mathbb{G}$ be a compact quantum group.
Let $c_c(\widehat{\mathbb{G}})$ be a subalgebra of $\mathcal{O}(\mathbb{G})'$ given by all of the linear functionals of the form $x \mapsto h(xy) \; (y\in \mathcal{O}(\mathbb{G})).$ Then $c_c(\widehat{\mathbb{G}})$ is a non-unital associative $*$-subalgebra. The Haar state $h$ can be regarded as a minimal self-adjoint central projection with $h\omega=\omega(1)h=\omega h$ for any $\omega \in c_c(\widehat{\mathbb{G}})$. 

Let $\mathcal{O}_c(\widehat{\mathcal{D}}(\mathbb{G}))$ be a $*$-algebra whose underlying vector space is $\mathcal{O}(\mathbb{G})\otimes c_c(\widehat{\mathbb{G}})$ and the following interchange law is satisfied:
\[\omega(\cdot x_{(2)})x_{(1)}=x_{(2)}\omega(\cdot x_{(1)}),\]
where the elementary tensor $x \otimes \omega$ is denoted by $x\omega$.
 
 It is known that $\mathcal{O}_c(\widehat{\mathcal{D}}(\mathbb{G}))$ can be made into a $*$-algebraic quantum group \cite[Theorem 3.16]{DVD} and we call the locally compact quantum group which has $\mathcal{O}_c(\widehat{\mathcal{D}}(\mathbb{G}))$ as its convolution algebra the Drinfeld double of $\mathbb{G}$. 
  
 For $\omega \in \mathcal{O}(\mathbb{G})'$, we define a linear functional $\widetilde \omega \in \mathcal{O}_c(\widehat{\mathcal{D}}(\mathbb{G}))'$ by $\widetilde \omega(x\theta)=\theta(1)\omega(x)$ where $x\in \mathcal{O}(\mathbb{G})$ and $\theta \in c_c(\widehat{\mathbb{G}})$.
 We write ${\rm Ind}$ for the embedding of $\mathcal{O}(\mathbb{G})'$ into $\mathcal{O}_c(\widehat{\mathcal{D}}(\mathbb{G}))'$ , $\omega \mapsto \widetilde{\omega} $. The image of ${\rm Ind}$ can be characterized as the set of elements $\widetilde{\omega}$ satisfying $\widetilde{\omega}(a)=\widetilde{\omega}(ah)$ for all $a\in \mathcal{O}_c(\widehat{\mathcal{D}}(\mathbb{G}))$.
 
 The following \cite[Theorem 29]{DCFY} gives the condition for states on $\mathcal{O}(\mathbb{G})$ to be central in terms of functionals on $\mathcal{O}_c(\widehat{\mathcal{D}}(\mathbb{G}))$.

 \begin{theorem}
A unital linear functional $\omega$ on $\mathcal{O}(\mathbb{G})$ is a central state if and only if $\widetilde{\omega} = {\rm Ind}(\omega)$ is positive on $\mathcal{O}_c(\widehat{\mathcal{D}}(\mathbb{G}))$.
\end{theorem}

Tube algebras are introduced by Ocneanu \cite[Section 3]{EK95}.
The tube algebra of $\mathcal{C}={\rm Rep}(\mathbb{G})$ is a space 
 \[ {\rm Tub}(\mathcal{C})= \bigoplus_{U_i, U_j\in {\rm Irr} (\mathbb{G})} {\rm Tub}(\mathcal{C})_{ij},\;\;\;\ {\rm Tub}(\mathcal{C})_{ij}=\bigoplus_{U_s\in {\rm Irr} (\mathbb{G})} {\rm Mor}(U_s\otimes U_j, U_i\otimes U_s)\]
 with the product and the involution defined as the following:
 \[(xy)^s_{ij}=\sum_{\substack{U_k, U_r, U_t\in {\rm Irr} (\mathbb{G}),\; \\ \omega \in {\rm onb}\; {\rm Mor}(U_s, U_r\otimes U_t)}}  (\iota_i\otimes \omega^*)(x^r_{ik}\otimes \iota_t)(\iota_r \otimes y^t_{kj})(\omega\otimes \iota_j),\]
 \[(x^*)^s_{ij}=(\bar{R}^*_s\otimes \iota_i\otimes\iota_s)(\iota_s\otimes(x^{\bar{s}}_{ji})^*\otimes\iota_s)(\iota_s\otimes\iota_j\otimes R_s).\]
 where $R_s\in {\rm Mor}(\mathbf{1}, \bar{U}_s\otimes U_s), \;\bar{R_s}\in {\rm Mor}(\mathbf{1}, U_s\otimes\bar{U}_s)$ are the solutions of the conjugate equations for $(U_s, \bar{U}_s)$ $i. e.$ they satisfy the following equations:
 \[(\iota_{\bar{X}}\otimes\bar{R_s}^*)(R_s\otimes\iota_{\bar{X}})=\iota_{\bar{X}},\;(\iota_X\otimes R_s^*)(\bar{R_s}\otimes\iota_X)=\iota_X.\]
 
 We consider the larger $*$-algebra defined as following:
 \[{\rm Tub}(\mathbb{G})=\bigoplus_{i,j}{\rm Tub}(\mathbb{G})_{i, j},\;\;\;{\rm Tub}(\mathbb{G})_{i, j}={\rm Tub}({\rm Rep}(\mathbb{G}))_{ij}\otimes B(H_{\bar{\jmath}}, H_{\bar{\imath}})\]
 where $H_{\bar{\jmath}}, H_{\bar{\imath}}$ are the conjugate spaces of $H_j, H_i$.
 The algebra structure is defined by that on ${\rm Tub}({\rm Rep}(\mathbb{G}))$ and the composition of operators between the spaces $H_k$. The involution is defined similarly.

 It is known that ${\rm Tub}({\rm Rep}(\mathbb{G}))$ is a full corner of ${\rm Tub}(\mathbb{G})$ \cite[Lemma 3.4]{NY18}. Namely ${\rm Tub}({\rm Rep}(\mathbb{G}))$ and ${\rm Tub}(\mathbb{G})$ are strongly Morita equivalent (cf. \cite[Example 3.6]{RW98}).
 
 The following theorem \cite[Theorem 3.5]{NY18} gives the isomorphism of ${\rm Tub}(\mathbb{G})$ and the Drinfeld double.
 
  \begin{theorem}
 We have an isomorphism of $*$-algebras ${\rm Tub}(\mathbb{G})\cong \mathcal{O}_c(\widehat{\mathcal{D}}(\mathbb{G})).$
 \end{theorem}

From the definition of the tube algebra and the monoidal equivalence between ${\rm Aut}^+(B)$ and $S_n^+$, we obtain ${\rm Tub}({\rm Rep}({\rm Aut}^+(B)))\cong {\rm Tub}({\rm Rep}(S_n^+))$. 
 Then we have that ${\rm Tub}({\rm Aut}^+(B))$ and ${\rm Tub}(S_n^+)$ are strongly Morita equivalent. It is known that there is the correspondence between the representations of strongly Morita eqivalent C*-algebras (see, for example \cite[Section 3.3]{RW98}). 
 Especially the projection of the trivial representation $p_{{\rm triv}}\in c_c(\widehat{{\rm Aut}^+(B)})$ corresponds to $p_{{\rm triv}}\in c_c(\widehat{S_n^+})$, where $p_{triv}$ is exactly the Haar state $h$ regarded as an element of $c_c(\widehat{\mathbb{G}})$. 
 Therefore by the condition of a linear functional on $\mathcal{O}(\mathbb{G})$ to be central, we obtain the one-to-one correspondence of central linear functionals on $\mathcal{O}({\rm Aut}^+(B))$ and $\mathcal{O}(S_n^+)$.
 By combining the discussion of the latter half of the previous section, we obtain $C^*(\chi_k : k\in \mathbb{N})\cong C([0,n])$. 
 
\bibliographystyle{plain}

\bibliography{biblio}

\end{document}